\documentclass[a4paper,12pt,oneside]{article}

\usepackage{amsmath}
\usepackage{amssymb}
\usepackage{amsthm}

\newtheorem{example}{Example}[section]

\newtheorem{theorem}[example]{Theorem}
\newtheorem{corollary}[example]{Corollary}

\newtheorem{lemma}[example]{Lemma}

\title{\textbf{On the Non-Commuting Graph of Dihedral Group}} 
\author{S.M.S.~Khasraw$^1$, I.D.~Ali$^2$ and R.R.~Haji$^3$\\
$^{1,2,3}$Department of Mathematics, College of Education,\\ Salahaddin University-Erbil, Erbil, Kurdistan Region, Iraq\\
$^1$sanhan.khasraw@su.edu.krd, $^2$evan.ali1@su.edu.krd,\\ $^3$rashad.haji@su.edu.krd}
\date{}

\begin{document}
\maketitle

\begin{center}
{\large{\textbf{Abstract}}}\\
\end{center}
For a nonabelian group G, the non-commuting graph $\Gamma_G$ of $G$ is defined as the graph with vertex set $G-Z(G)$, where $Z(G)$ is the center of $G$, and two distinct vertices of $\Gamma_G$ are adjacent if they do not commute in $G$. In this paper, we investigate the detour index, eccentric connectivity and total eccentricity polynomials of non-commuting graph on $D_{2n}$. We also find the mean distance of non-commuting graph on $D_{2n}$. 

\section{Introduction}
The concept of non-commuting graph of a finite group has been introduced by Abdollahi \textit{et al} in 2006 \cite{MR0294487}. For a non-abelian group $G$, associate a graph $\Gamma_G$ with it such that the vertex set of $\Gamma_G$ is $G-Z(G)$, where $Z(G)$ is the center of $G$, and two distinct vertices $x$ and $y$ are adjacent if they don't commute in $G$, that is, $xy \neq yx$. Several works on assigning a graph to a group and investigation of algebraic properties of group using the associated graph have been done, for example, see \cite{bertram1983some, bertram1990graph, ali2016commuting}.\\
All graphs are considered to be simple, which are undirected with no loops or multiple edges. Let $\Gamma$ be any graph, the sets of vertices and edges of $\Gamma$ are denoted by $V(\Gamma)$ and $E(\Gamma)$, respectively. The cardinality of the vertex set $V(\Gamma)$ is called the \textit{order} of the graph $\Gamma$ and is denoted by $|V(\Gamma)|$ and the number of edges of the graph $\Gamma$ is called the \textit{size} of $\Gamma$, and denoted by $|E(\Gamma)|$. The graph $\Gamma$ is called \textit{split} if $V(\Gamma)=S\cup K$, where $S$ is an independent set and $K$ is a complete set. For a vertex $v$ in $\Gamma$, the number of edges incident to $v$ is called the \textit{degree} of $v$ and is denoted by $deg_\Gamma(v)$. The \textit{eccentricity} of a vertex $v$ in $\Gamma$, denoted by $ecc(v)$, is the largest distance between $v$ and any other vertex $u$ in $\Gamma$. For vertices $u$ and $v$ in a graph $\Gamma$, a $u-v$ \textit{path} in $\Gamma$ is $u-v$ walk with no vertices repeated. The shortest (longest) $u-v$ path in a graph $\Gamma$, denoted by $d(u, v)$ $(D(u, v))$, is called the \textit{distance(detour distance)} between vertices $u$ and $v$ in $\Gamma$. The \textit{detour index}, \textit{eccentric connectivity }and \textit{total eccentricity polynomials} are defined as $D(\Gamma_\Omega, x)=\sum_{u,v \in V(\Gamma)}x^{D(u,v)}$ \cite{shahkoohi2011polynomial}, $\Xi(\Gamma, x)=\sum_{u \in V(\Gamma)}deg_\Gamma(u)x^{ecc(u)}$ and $\Theta(\Gamma, x)=\sum_{u \in V(\Gamma)}x^{ecc(u)}$ \cite{dovslic2011eccentric}, respectively. The \textit{detour index} $dd(\Gamma)$, the \textit{eccentric connectivity index} and the \textit{total eccentricity} $\xi^c(\Gamma)$ of a graph $\Gamma$ are the first derivatives of their corresponding polynomials at $x=1$, respectively. A \textit{transmission} of a vertex $v$ in $\Gamma$ is $\sigma (v,\Gamma)=\Sigma_{u\in V(\Gamma)}d(u,v)$. The transmission of a graph $\Gamma$ is $\sigma (\Gamma)=\Sigma_{u\in V(\Gamma)}\sigma (u,\Gamma)$. The \textit{mean(average) distance} of a graph $\Gamma$ is $\mu (\Gamma)=\frac{ \sigma (\Gamma)}{p(p-1)}$, where $p$ is the order of $\Gamma$, see \cite{Gr123, Gr456, Gr789}. In this paper, we study some properties of non-commuting graph of dihedral groups. The dihedral group $D_{2n}$ of order $2n$ is defined by $$D_{2n}=\langle r, s : r^n = s^2 =1, srs = r^{-1} \rangle$$ for any $n \geq 3$, and the center of $D_{2n}$ is
$
Z(D_{2n})=\left\{
\begin{tabular}{ll}
$\{1\}$, & \mbox{ if }$n$ is odd\\
$\{1, r^\frac{n}{2}\}$, & \mbox{ if }$n$ is even.\\
\end{tabular}
\right.
$
Throughout this article, we assume that $\Omega_1=\{r^i : 1 \leq i \leq n\}- Z(D_{2n})$, and $\Omega_2 = \{sr^i : 1 \leq i \leq n \}$. 
This article is organized as follows: In the present section, we give some important definitions and notations. In Section 2, we study some basic properties of the non-commuting graph $\Gamma_{D_{2n}}$ of $D_{2n}$. We see that $\Gamma_{D_{2n}}$ is a split graph if $n$ is an odd integer.\\
In Section 3, we find the detour index, eccentric connectivity and total eccentricity polynomials of the non-commuting graph $\Gamma_{D_{2n}}$. In Section 4, we find the mean distance of the graph $\Gamma_{D_{2n}}$.       

\section{Some properties of the non - commuting graph of $D_{2n}$}

Recall that, for any $n \geq 3$, $D_{2n}=\langle r, s : r^n = s^2 =1, srs = r^{-1} \rangle$, $\Omega_1=\{r^i : 1 \leq i \leq n\}- Z(D_{2n})$, and $\Omega_2 = \{sr^i : 1 \leq i \leq n \}$.  

\vspace{.5cm}

We start with the following lemma, which has been proved in \cite{MR0294487}.

\begin{lemma}\label{lem21}
Let $G$ be any non-abelian finite group and $a$ be any vertex of $\Gamma_G$. Then $deg_{\Gamma_G}(a)=|G|-|C_G(a)|$, where $C_G(a)$ is the centralizer of the element $a$ in the group $G$.  
\end{lemma}

According to the above lemma, we can state the following.

\begin{theorem}\label{thm22} 
In the graph $\Gamma_\Omega$, where $\Omega=\Omega_1 \cup \Omega_2$, we have \\
$1.\; deg_{\Gamma_\Omega}(r^i)=n$ for any $n$,\\
$
2. \;deg_{\Gamma_\Omega}(sr^i)=\left\{
\begin{tabular}{ll}
$2n-2$, & \mbox{ if }$n$ is odd\\
$2n-4$, & \mbox{ if }$n$ is even.\\
\end{tabular}
\right.
$
\end{theorem}
\begin{proof}
1. Since $C_{D_{2n}}(r^i)=\{r^i : 1 \leq i \leq n \}$, then, from Lemma \ref{lem21}, $deg_{\Gamma_\Omega}(r^i) = |D_{2n}|-|C_{D_{2n}}(r^i)|= 2n-n = n$.\\
2. If $n$ is odd, then $C_{D_{2n}}(sr^i) = \{1, sr^i\}$ for all $i$, $1 \leq i \leq n$. This follows that $deg_{\Gamma_\Omega}(sr^i) = 2n-2$ for all $1 \leq i \leq n$. If $n$ is even, then $C_{D_{2n}}(sr^i) = \{1, r^{\frac{n}{2}}, sr^i, sr^{\frac{n}{2}+i}\}$ for all $1 \leq i \leq n$. Thus, $deg_{\Gamma_\Omega}(sr^i)= 2n-4$ for all $1 \leq i \leq n$.   
\end{proof}

\medskip

\begin{theorem}\label{thm23}
Let $\Gamma_\Omega$ be a non-commuting graph on $D_{2n}$.
\begin{enumerate}
\item If $\Omega=\Omega_1$, then $\Gamma_\Omega=\overline{K}_l$, where $l=|\Omega_1|$.
\item If $\Omega=\Omega_2$, then

$\Gamma_\Omega=\left\{
\begin{tabular}{ll}
$K_n$, & \mbox{ if }$n$ is odd\\
$K_n-\frac{n}{2}K_2$, & \mbox{ if }$n$ is even.\\
\end{tabular}
\right.
$\\

where $\frac{n}{2}K_2$ denotes $\frac{n}{2}$ copies of $K_2$. 
\end{enumerate} 
\end{theorem}
\begin{proof}
1. The centralizer of $r^i$, $1 \leq i \leq n$, is $C_{D_{2n}}(r^i)=\{r^i : 1 \leq i \leq n \}$ of size $n$, then there is no edge between any pair of vertices in $\Gamma_{\Omega_1}$. Thus, $\Gamma_{\Omega_1}=\overline{K}_l$, where $l=|\Omega_1|$.\\ 
2. \textbf{When $n$ is odd}. Since the element $sr^i$, where $i=1,2,...,n$, has centralizer $C_{D_{2n}}(sr^i)=\{1, sr^i\}$ of size 2, so let $\Omega=\Omega_2=\{sr, sr^2, ..., sr^n\}$. Then the subgraph $\Gamma_{\Omega}=K_n$ is complete.\\
\textbf{When $n$ is even}. Since $C_{D_{2n}}(sr^i) = \{1, r^{\frac{n}{2}}, sr^i, sr^{\frac{n}{2}+i}\}$ for all $1 \leq i \leq n$. Then there is no edge between the vertices $sr^i$ and $sr^{\frac{n}{2}+i}$ in $\Gamma_\Omega$ for all $1 \leq i \leq n$. Therefore, $\Gamma_\Omega= K_n-\frac{n}{2}K_2$    
\end{proof}

\medskip

\begin{theorem}\label{thm24}
Let $n\geq 3$ be an odd integer and $H$ be a subset of $D_{2n}-Z(D_{2n})$. Then $\Gamma_H=K_{1, n-1}$ if and only if $H=\{sr^i, r, r^2, \cdots, r^{n-1}\}$ for some $i$. 
\end{theorem}
\begin{proof}
Suppose that $\Gamma_H=K_{1, n}$. By Theorem \ref{thm22}, $H=\{sr^i, r, r^2, \cdots, r^{n-1}\}$ for some $i$. Conversely, suppose $H=\{sr^i, r,  r^2, \cdots, r^{n-1}\}$. Then $C_H(sr^i)=\{sr^i\}$ and $C_H(r^j)=\{r, r^2, \cdots, r^{n-1}\}$ for $1 \leq j \textless n$. Thus, $\Gamma_H=K_{1, n-1}$. 
\end{proof}

\medskip

\begin{corollary}
Let $n\geq 3$ be an odd integer and $\Omega=\Omega_1 \cup \Omega_2$. Then $\Gamma_\Omega$ is a split graph.
\end{corollary}
\begin{proof}
The proof follows from Theorem \ref{thm23} and Theorem \ref{thm24}.
\end{proof}

\medskip

\begin{theorem}
Let $\Gamma_\Omega$ be a non-commuting graph on $D_{2n}$, where $\Omega=\Omega_1 \cup \Omega_2$. We have
$$
|E(\Gamma_\Omega)|=\left\{
\begin{tabular}{ll}
$\frac{3n(n-1)}{2}$, & \mbox{ if }$n$ is odd;\\
$\frac{3n(n-2)}{2}$, & \mbox{ if }$n$ is even.\\
\end{tabular}
\right.
$$
\end{theorem}
\begin{proof}
It is clear that $\Omega_1 \cap \Omega_2 =\emptyset$ and $\Omega_1 \cup \Omega_2 =D_{2n}-Z(D_{2n})=\Omega$. According to $n$, there are two cases to consider.\\
\textbf{Case 1.} If $n$ is odd, then the subgraph induced by $\Omega_1$ has no edges and the subgraph induced by $\Omega_2$ is complete. Thus, the number of edges in $\Gamma_\Omega$ is sum of the number of edges in $\langle \Omega_2 \rangle$ and the number of edges from set of vertices in $\Omega_1$ to set of vertices in $\Omega_2$. Therefore, $|E(\Gamma_\Omega)|=\frac{n(n-1)}{2}+
n(n-1)=\frac{3n(n-1)}{2}$.\\ 
\textbf{Case 2.} If $n$ is even, then the subgraph induced by $\Omega_1$ has no edges and the subgraph induced by $\Omega_2$ has $\frac{n(n-1)}{2}-\frac{n}{2}=\frac{n(n-2)}{2}$ edges. Thus, the number of edges in $\Gamma_\Omega$ is sum of the number of edges in $\langle \Omega_2 \rangle$ and the number of edges from set of vertices in $\Omega_1$ to set of vertices in $\Omega_2$. Therefore, $|E(\Gamma_\Omega)|=\frac{n(n-2)}{2}+
n(n-2)=\frac{3n(n-2)}{2}$.   
\end{proof}

\section{Detour index, eccentric connectivity and total eccentricity polynomials of non- commuting graphs on $D_{2n}$}

\begin{theorem}\label{thm31}
Let $\Gamma_\Omega$ be a non-commuting graph on $D_{2n}$, where $\Omega=\Omega_1 \cup \Omega_2$. Then for any $u, v \in \Gamma_\Omega$,
$$
D(u, v)=\left\{
\begin{tabular}{ll}
$2n-2$, & \mbox{ if }$n$ is odd;\\
$2n-3$, & \mbox{ if }$n$ is even.\\
\end{tabular}
\right.
$$
\end{theorem}
\begin{proof}
There are two cases. \textbf{When $n$ is odd}. From Theorem \ref{thm23} and Theorem \ref{thm24}, we see that no two vertices in $\Omega_1$ are adjacent, any pair of distinct vertices in $\Omega_2$ are adjacent, and each vertex in $\Omega_1$ is adjacent to every vertex in $\Omega_2$. Then for all $u, v \in \Omega$, there is a $u - v$ path of length $2n-2$.\\
\textbf{When $n$ is even}. Again, no two vertices in $\Omega_1$ are adjacent, each vertex in $\Omega_1$ is adjacent to every vertex in $\Omega_2$, and any pair of distinct vertices $u$ and $v$ in $\Omega_2$ are adjacent if $u, v \notin \{sr^i, sr^{\frac{n}{2}+i}\}$ for $1 \leq i \leq \frac{n}{2}$. So, for all $u, v \in \Omega$, there is a $u - v$ path of length $2n-3$.
\end{proof}

\medskip

\begin{theorem}\label{thm32}
Let $\Gamma_\Omega$ be a non-commuting graph on $D_{2n}$, where $\Omega=\Omega_1 \cup \Omega_2$. Then
$$
D(\Gamma_\Omega, x)=\left\{
\begin{tabular}{ll}
$(n-1)(2n-1)x^{2n-2}$, & \mbox{ if }$n$ is odd;\\
$(n-1)(2n-3)x^{2n-3}$, & \mbox{ if }$n$ is even.\\
\end{tabular}
\right.
$$
\end{theorem}
\begin{proof}
\textbf{Case 1}. n is odd. Since $|\Gamma_\Omega|=2n-1$, there are ${{2n-1}\choose{2}}=(n-1)(2n-1)$ possibilities of distinct pairs of vertices. By Theorem \ref{thm31}, $D(u, v)=2n-2$ for any $u, v \in \Gamma_\Omega$. Then $D(\Gamma_\Omega, x)=\sum_{\{u,v\}}x^{D(u,v)}={{2n-1}\choose{2}}x^{2n-2}=(n-1)(2n-1)x^{2n-2}$.\\ 
\textbf{Case 2}. n is even. We have that $|\Gamma_\Omega|=2n-2$ and the possibility of taking distinct pairs of vertices form $\Gamma_\Omega$ is ${{2n-2}\choose{2}}=(n-1)(2n-3)$. From Theorem \ref{thm31}, we deduce that $D(\Gamma_\Omega, x)=\sum_{\{u,v\}}x^{D(u,v)}={{2n-2}\choose{2}}x^{2n-3}=(n-1)(2n-3)x^{2n-3}$.   
\end{proof}

\begin{corollary}
For the graph $\Gamma_\Omega$,
$$
dd(\Gamma_\Omega)=\left\{
\begin{tabular}{ll}
$2(n-1)^2(2n-1)$, & \mbox{ if }$n$ is odd;\\
$(n-1)(2n-3)^2$, & \mbox{ if }$n$ is even.\\
\end{tabular}
\right.
$$
\end{corollary}
\begin{proof}
It is clear that $dd(\Gamma_\Omega)=\frac{d}{dx}(D(\Gamma_\Omega, x))|_{x=1}$. From Theorem \ref{thm32}, the result follows.
\end{proof}

\begin{theorem}\label{thm34}
Let $\Gamma_\Omega$ be a non-commuting graph on $D_{2n}$, where $\Omega=\Omega_1 \cup \Omega_2$.
\begin{enumerate}
\item When $n$ is odd, then 
$$
ecc(v)=\left\{
\begin{tabular}{ll}
$2$, & \mbox{ if }$v \in \Omega_1$;\\
$1$, & \mbox{ if }$v \in \Omega_2$.\\
\end{tabular}
\right.
$$
\item When $n$ is even, then $ecc(v)=2$ for each $v \in \Omega$.
\end{enumerate}
\end{theorem}
\begin{proof}
1. \textbf{When $n$ is odd}. There is no edge between any pair of vertices in $\Omega_1$ and each vertex in $\Omega_2$ is adjacent to every vertex in $\Omega$. So the maximum distance between any vertex of $\Omega_1$ and the other vertices in $\Omega$ is 2 and the maximum distance between any vertex of $\Omega_2$ and the other vertices in $\Omega$ is 1.\\
2. \textbf{When $n$ is even}. Again, There is no edge between any pair of vertices in $\Omega_1$. Also, each vertex in $\Omega_1$ is adjacent to every vertex in $\Omega_2$. Thus, $ecc(v)=2$ for each $v \in \Omega_1$. By Theorem \ref{thm23}, the subgraph $\Gamma_{\Omega_2}$ is not a complete graph because there is no edge between the vertices $sr^i$ and $sr^{i+\frac{n}{2}}$. This means that the maximum distance between any vertex in $\Omega_2$ and any other vertex in $\Omega$ is 2, so $ecc(v)=2$ for each $v \in \Omega_2$.   
\end{proof}

From above theorem, we can have the following.

\begin{theorem}
Let $\Gamma_\Omega$ be a non-commuting graph on $D_{2n}$, where $\Omega=\Omega_1 \cup \Omega_2$. Then
\begin{enumerate}
\item 
$$
\Xi(\Gamma_\Omega, x)=\left\{
\begin{tabular}{ll}
$n(n-1)x^2+2n(n-1)x$, & \mbox{ if }$n$ is odd;\\
$3n(n-2)x^2$, & \mbox{ if }$n$ is even.\\
\end{tabular}
\right.
$$
\item 
$$
\Theta(\Gamma_\Omega, x)=\left\{
\begin{tabular}{ll}
$(n-1)x^2+nx$, & \mbox{ if }$n$ is odd;\\
$2(n-1)x^2$, & \mbox{ if }$n$ is even.\\
\end{tabular}
\right.
$$
\end{enumerate}
\end{theorem}
\begin{proof}
The proof follows directly from Theorem \ref{thm22} and Theorem \ref{thm34}.
\end{proof}
From the above theorem, one can obtain the eccentric connectivity index and the total eccentricity of a graph $\Gamma_\Omega$ from their corresponding polynomials by computing their first derivatives at $x = 1$.

\begin{corollary}
Let $\Gamma_\Omega$ be a non-commuting graph on $D_{2n}$, where $\Omega=\Omega_1 \cup \Omega_2$. Then
$$
\xi^c(\Gamma_\Omega)=\left\{
\begin{tabular}{ll}
$4n(n-1)$, & \mbox{ if }$n$ is odd;\\
$6n(n-2)$, & \mbox{ if }$n$ is even.\\
\end{tabular}
\right.
$$
\end{corollary}

\section{The mean distance of the graph $\Gamma_{D_{2n}}$}

Through this section we find the mean (average) distance of the graph $\Gamma_{D_{2n}}$. 

\begin{lemma}\label{lemma1}
In the graph $\Gamma_{D_{2n}}$, where $n$ is odd, the transmission of each vertex $r^i$ is $\sigma (r^i, \Gamma_{D_{2n}})=3n-4$ for all
$1\leq i\leq n-1$ and the transmission of a vertex $sr^i$ is 
$\sigma (sr^i,\Gamma_{D_{2n}})=2n-2$ for all $1\leq i\leq n$.
\end{lemma}
\begin{proof}
The vertices set of the graph $\Gamma_{D_{2n}}$ is 
$V(\Gamma_{D_{2n}})=\{r^i,sr^j: 1\leq i< n, 1\leq j\leq n\}$. 
Then $|V(\Gamma_{D_{2n}})|=2n-1$, where $n$ is odd. 
A vertex $r^i$ is adjacent with all vertices $sr^j$ for all 
$1\leq j\leq n$, so, $d(r^i,sr^j)=1$ for all $1\leq i\leq n-1$
and all $1\leq j\leq n$. While a vertex $r^i$ is not adjacent with 
$r^j$ for all $i\ne j$, $1\leq i\leq n-1$ and $1\leq j\leq n$, 
then $d(r^i,r^j)=2$ for all $1\leq i\leq n-1$, $1\leq j\leq n$
 and $i\ne j$. So, 
$$\sigma (r^i,\Gamma_{D_{2n}})=\Sigma_{\substack{1\leq j<n\\j\ne i}}d(r^i,r^j)+\Sigma_{1\leq j\leq n}d(r^i,sr^j)=2(n-2)+n=3n-4$$
for all $1\leq i\leq n-1$.
On the other hand every vertex $sr^i$ is adjacent with $sr^j$ for all $i\ne j$, $1\leq i,j\leq n$. Therefore, $d(sr^i,sr^j)=1$, for all $i\ne j$, $1\leq i,j\leq n$. Also, every vertex $sr^i$ is adjacent with $r^j$, then $d(sr^i, r^j)=1$ for all $1\leq i\leq n$, 
$1\leq j\leq n-1$. So, 
$$\sigma (sr^i,\Gamma_{D_{2n}})=\Sigma_{\substack{1\leq i,j\leq n\\i\ne j}}d(sr^i,sr^j)+\Sigma_{1\leq j<n}d(sr^i,r^j)=(n-1)+(n-1)=2n-2,$$
for all $1\leq i\leq n$.
\end{proof}

\begin{lemma}\label{lemma2}
In the graph $\Gamma_{D_{2n}}$, where $n$ is even, the transmission of each vertex $r^i$ is $\sigma (r^i, \Gamma_{D_{2n}})=3n-6$ for all
$1\leq i\leq n-1$ and the transmission of a vertex $sr^i$ is 
$\sigma (sr^i,\Gamma_{D_{2n}})=2n-2$ for all $1\leq i\leq n.$
\end{lemma}
\begin{proof} Let $M=\{1,2,\ldots,n-1\}-\{n/2\}$. Then the vertices set of the graph $\Gamma_{D_{2n}}$, where $n$ is even, is 
$V(\Gamma_{D_{2n}})=\{r^i,sr^j: i\in M, 1\leq j\leq n\}$. 
So, $|V(\Gamma_{D_{2n}})|=2n-2$. A vertex $r^i$ is adjacent with all vertices $sr^j$ for all $i\in M$ and all $1\leq j\leq n$. Thus, $d(r^i,sr^j)=1$ for all $i\in M$ and all $1\leq j\leq n$. 
Notice that every two vertices $r^i$ and $r^j$ are non-adjacent 
for all $i,j\in M$ and $i\ne j$, then $d(r^i,r^j)=2$ 
for all $i,j\in M$ and $i\ne j$. So,
$$\sigma (r^i,\Gamma_{D_{2n}})=\Sigma_{\substack{j\in S\\j\ne i}}d(r^i,r^j)+\Sigma_{1\leq j\leq n}d(r^i,sr^j)=2(n-3)+n=3n-6$$
for all $i\in M$.
Also, every vertex $sr^i$ is adjacent with $sr^j$ for all 
$i\ne j$, $1\leq i\leq n/2$, and all $j\in \{1,2,\ldots,n-1\}-\{i+n/2\}$, 
then $d(sr^i,sr^j)=1$, for all $j\in \{1,2,\ldots,n-1\}-\{i+n/2\}$, 
and $d(sr^i,sr^{i+n/2})=2$, for all $1\leq i\leq n/2$.
Since each vertex $sr^i$ is adjacent with all vertices $r^j$, for all $1\leq i\leq n$, 
and $j\in M$, then $d(sr^i, r^j)=1$. Therefore, 
$$\sigma (sr^i,\Gamma_{D_{2n}})=\Sigma_{\substack{1\leq j\leq n\\j\ne i}}d(sr^i,sr^j)+\Sigma_{j\in S}d(sr^i,r^j)=(n-2)+2+(n-2)=2n-2,$$
for all $1\leq i\leq n$.
\end{proof}

\begin{theorem}
The mean distance of the graph $\Gamma_{D_{2n}}$, where $n$ is odd, is $\mu (\Gamma_{D_{2n}})=\frac{5n-4}{4n-2}$.
\end{theorem}
\begin{proof}
By Lemma~\ref{lemma1}, we see that the transmission of 
the graph $\Gamma_{D_{2n}}$ is
\begin{align*}
\sigma(\Gamma_{D_{2n}})&=\Sigma_{i=1}^{n-1}\sigma(r^i,\Gamma_{D_{2n}})+\Sigma_{i=1}^n\sigma(sr^i,\Gamma_{D_{2n}})\\
&=(n-1)(3n-4)+n(2n-2)\\
&=5n^2-9n+4.
\end{align*} 
Notice that $|V(\Gamma_{D_{2n}})|=2n-1$.
Therefore, $\mu(\Gamma_{D_{2n}})=\frac{\sigma(\Gamma_{D_{2n}})}{ |V(\Gamma_{D_{2n}})|(|V(\Gamma_{D_{2n}})|-1)}
=\frac{5n^2-9n+4}{(2n-1)(2n-2)}=\frac{5n-4}{4n-2}$.
\end{proof}

\begin{theorem}
The mean distance of the graph $\Gamma_{D_{2n}}$, where $n$ is even, is $\mu (\Gamma_{D_{2n}})=\frac{5n^2-14n+12}{(2n-2)(2n-3)}$.
\end{theorem}
\begin{proof}
By using Lemma~\ref{lemma2}, we can find the transmission of 
the graph $\Gamma_{D_{2n}}$ which is
\begin{align*}
\sigma(\Gamma_{D_{2n}})&=\Sigma_{\substack{i=1\\i\ne n/2}}^{n-1}\sigma(r^i,\Gamma_{D_{2n}})+\Sigma_{i=1}^n\sigma(sr^i,\Gamma_{D_{2n}})\\
&=(n-2)(3n-6)+n(2n-2)\\
&=5n^2-14n+12.
\end{align*} 
Notice that $|V(\Gamma_{D_{2n}})|=2n-2$.
Therefore, $\mu(\Gamma_{D_{2n}})=\frac{\sigma(\Gamma_{D_{2n}})}{ |V(\Gamma_{D_{2n}})|(|V(\Gamma_{D_{2n}})|-1)}
=\frac{5n^2-14n+12}{(2n-2)(2n-3)}$.
\end{proof}

\end{document}